\documentclass{amsart}

\usepackage{amsmath, amsthm, amssymb, amsfonts, amsmath,enumerate,fullpage}
\usepackage{graphicx, tikz}
\usetikzlibrary{arrows,positioning,matrix}

\newcommand{\Span}{\operatorname{span}}
\newcommand{\diag}{\operatorname{diag}}
\newcommand{\Id}{\operatorname{Id}}
\newtheorem{theorem}{Theorem}[section]

\newtheorem{lemma}{Lemma}[section]
\theoremstyle{definition}

\newtheorem{remark}{Remark}[section]

\author{Nicholas F. Marshall}
\address{Department of Mathematics, Yale University, 
New Haven, CT 06511, USA}
\title{The Stability of the First Neumann Laplacian Eigenfunction Under Domain
Deformations and Applications}


\begin{document}

\begin{abstract}
The robustness of manifold learning methods is often predicated on the stability
of the Neumann Laplacian eigenfunctions under deformations of the assumed
underlying domain.  Indeed, many manifold learning methods are based on
approximating the Neumann Laplacian eigenfunctions on a manifold that is
assumed to underlie data, which is viewed through a source of distortion. In
this paper, we study the stability of the first Neumann Laplacian eigenfunction
with respect to deformations of a domain by a diffeomorphism.  In particular,
we are interested in the stability of the first eigenfunction on tall thin
domains where, intuitively, the first Neumann Laplacian eigenfunction should
only depend on the length along the domain. We prove a rigorous version of this
statement and apply it to a machine learning problem in geophysical
interpretation. 
\end{abstract}

\maketitle

%
%
\section{Introduction and Main Result}
\subsection{Motivation}
We are motivated by a machine learning problem in the field of geophysical
interpretation: given a seismic image, the objective is to reparameterize depth
in the image (the $y$-axis) such that each layer in the seismic image has
constant depth. We propose an unsupervised diffusion based method
to achieve this goal. In particular, we propose defining a discrete diffusion
process that travels rapidly along the layers of a seismic image, and slowly
perpendicular to the layers. This anisotropic diffusion process corresponds to
an isotropic diffusion process on some tall thin domain formed by contracting
the metric of the seismic image along its layers, see the illustration in Figure
\ref{fig0}.

\begin{figure}[h!]
\centering
\begin{tabular}{ccc}
\raisebox{-.5\height}{
\includegraphics[height=.11\textheight]{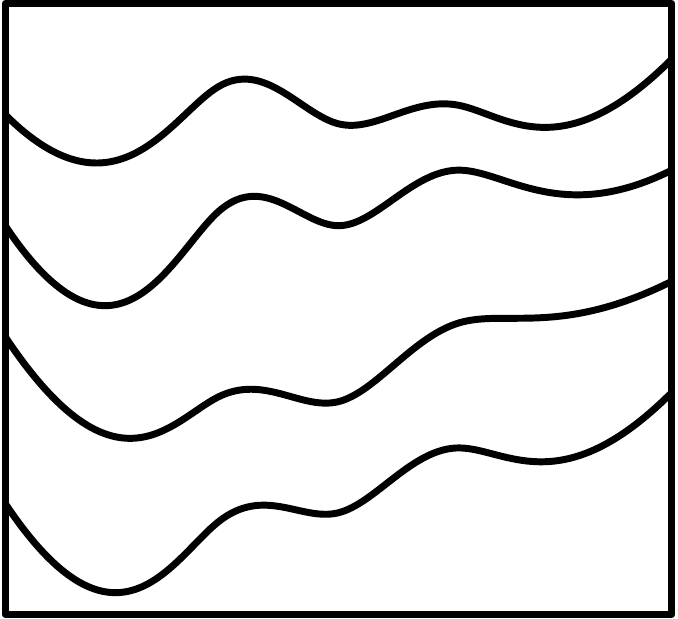}} &
$\left( \substack{\text{underlying} \\\text{tall thin} \\ \text{domain}}
\raisebox{-.5\height}{
\includegraphics[height=.11\textheight]{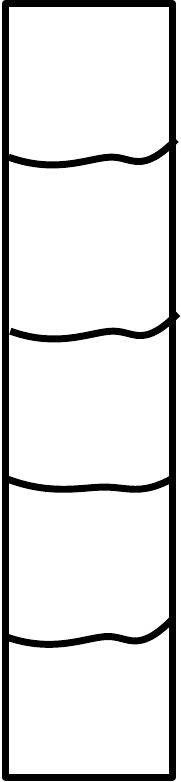}}
\qquad
\right)$ &
$\overset{\text{output}}{\mapsto}$
\raisebox{-.5\height}{
\includegraphics[height=.11\textheight]{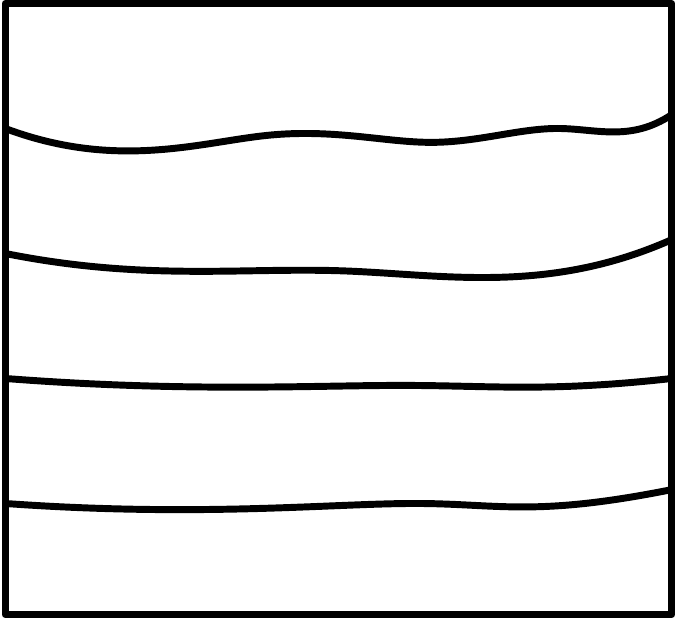}}
\end{tabular}
\caption{Cartoon of a seismic image (left), underlying tall thin domain
(center), and reparameterized image (right).} \label{fig0}
\end{figure}

The first eigenfunction of this anisotropic diffusion operator is
then used to
reparameterize depth in the seismic image. Specifically, if the diffusion
operator is defined with Neumann boundary conditions (as is standard for
diffusion based machine learning methods \cite{Lafon:2004}), then the first
eigenfunction of the diffusion operator will approximate the first Neumann
Laplacian eigenfunction of the underlying tall thin domain. In the ideal case,
this underlying domain will be a tall rectangle, whose first Neumann Laplacian
eigenfunction is $\cos(\pi y/h)$,  where $h$ is the height of the rectangle. We
can use this eigenfunction to assign each pixel in the seismic image its height
in the underlying tall rectangle, thereby flattening the layers of the image.

\subsection{Stability of the first Neumann eigenfunction}
The main challenge in the proposed algorithm is constructing the
anisotropic diffusion process, which rapidly travels along the layers of the
given seismic image. Inaccuracies in the construction of this anisotropic
diffusion process, will cause the underlying domain on which the anisotropic
diffusion process is an isotropic diffusion process to be a deformed version of
a tall rectangle. This issue motivates the following question: how
stable is the first Neumann Laplacian eigenfunction under domain deformations?
Numerically, we observe that on long thin domains, the Neumann Laplacian
eigenfunctions demonstrate a high degree of stability with respect to
domain deformations. To illustrate this idea, we compute the first
Neumann Laplacian eigenfunction numerically on two different domains, and we
visualize each eigenfunction on its respective domain using a color map, see
Figure \ref{fig:firsteigen}.

\begin{figure}[ht!]
\centering
\begin{tabular}{cc}
$\Omega_1$ & $\Omega_2$ \\
\raisebox{-.5\height}{\includegraphics[width = .27\textwidth]{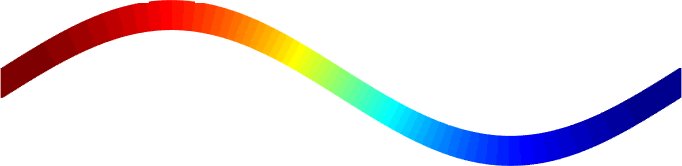}} & 
\raisebox{-.5\height}{\includegraphics[width = .27\textwidth]{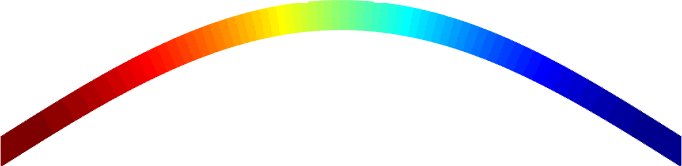}} \\
\end{tabular} 
\caption{The first Neumann Laplacian eigenfunction on two different long thin
domains.}
\label{fig:firsteigen}
\end{figure}

To formalize a notion of stability, let $R \subset \mathbb{R}^d$
be an open bounded connected domain with a piecewise smooth boundary $\partial
R$. Suppose that $R$ is transformed by a diffeomorphism $\varphi$ into a domain $\Omega :=
\varphi(R)$.  Let $0 = \eta_0 < \eta_1 \le \eta_2 \le \cdots$ and
$v_0,v_1,v_2,\ldots$ denote the Neumann Laplacian eigenvalues and eigenfunctions
on $R$, and let $0 = \mu_0 < \mu_1 \le \mu_2 \le \cdots$ and $u_0, u_1, u_2,
\ldots$ denote the Neumann Laplacian eigenvalues and eigenfunctions on $\Omega$.
More precisely, $\eta_j,v_j$ and $\mu_j,u_j$ are characterized by the following
eigenvalue problems
$$
\left\{ \begin{array}{cc}
-\Delta v_j = \eta_j v_j & \text{in } R \\
\frac{\partial}{\partial n} v_j = 0 & \text{on } \partial R
\end{array} \right.
\quad \text{and} \quad
\left\{ \begin{array}{cc}
-\Delta u_j = \mu_j u_j & \text{in } \Omega \\
\frac{\partial}{\partial n} u_j = 0 & \text{on } \partial \Omega,
\end{array} \right.
$$
where $\Delta$ denotes the Laplacian, and $n$ denotes an outward normal to the
boundary of the domain. We can now formulate precise questions about the
stability of the first Neumann Laplacian eigenfunction. For example, we can
consider the distance between $u_1 \circ \varphi$ and $v_1$ with respect to
the $L^2(R)$-norm. More generally, we can ask: how accurately can we represent
$u_1 \circ \varphi$ in the span of the first $k$ eigenfunctions
$v_1,\ldots,v_k$?  Quantitative answers to these questions will provide some
understanding of the stability of the first Neumann eigenfunction that we
observe both numerically and in machine learning applications.

\subsection{Preview of application results}
The stability of the first Neumann eigenfunction is also evident in the
application results. Given a seismic image, we construct a
discrete diffusion process that travels rapidly along the layers of the image as
described in \S \ref{sec:application}. Given the high level of noise in the
image, this anisotropic diffusion process cannot be defined perfectly,
yet  when we reparameterize depth in the image using the first eigenfunction of
the diffusion, the layers are roughly flat, see Figure \ref{fig:flattening2}.
Moreover, the reparameterized image actually encodes more detailed geometric
information: by zooming into the interval $[0.6,1.2]$ on the $y$-axis of the
reparameterized space, we observe that the individual layers of the image have
been automatically separated by the reparameterization, which makes it possible
to use post-processing to further refine the results. 

\begin{figure}[h!]
\centering
\begin{tabular}{cc}
\raisebox{-.5\height}{ \includegraphics[height=0.17\textwidth]{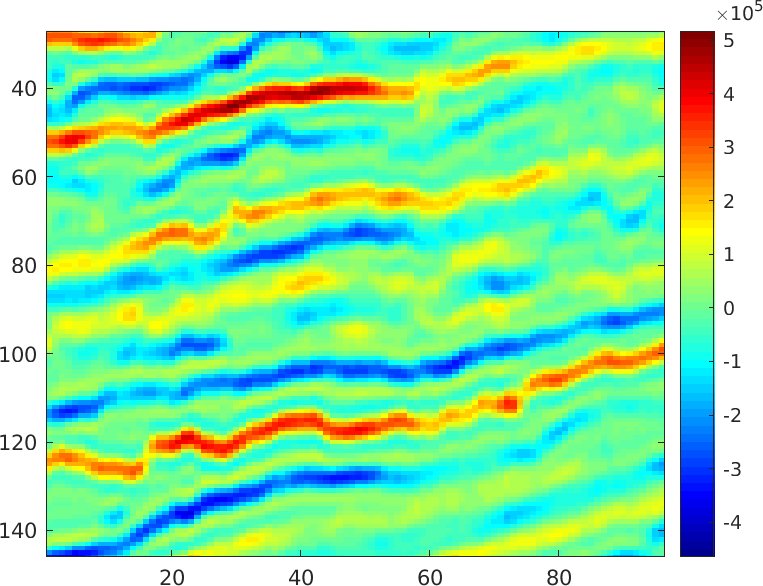}} { $\overset{\text{output}}{\mapsto}$}
\raisebox{-.5\height}{ \includegraphics[height=0.17\textwidth]{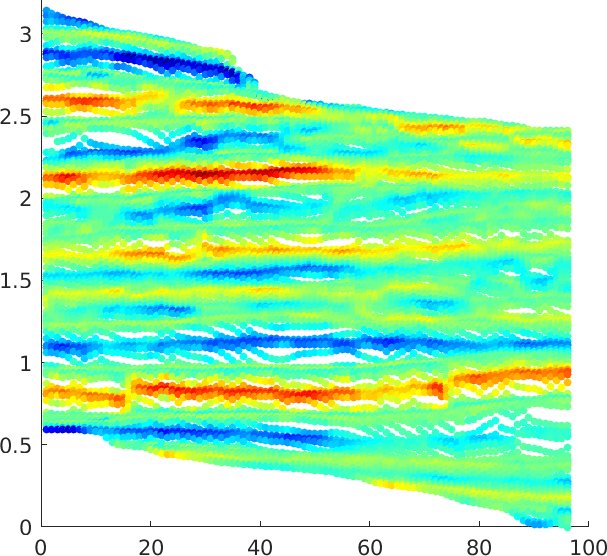}}
 $ \substack{\text{zoomed} \\ \text{image}}
\raisebox{-.5\height}{ \includegraphics[height=0.17\textwidth]{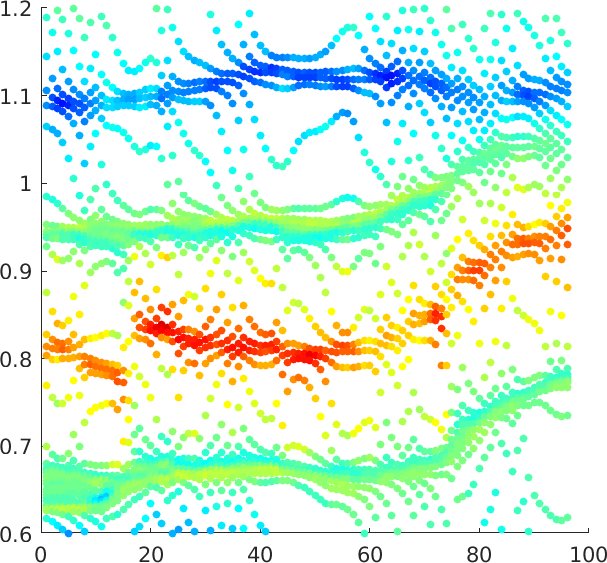}}
$
\end{tabular}

\caption{Seismic image (left), result of diffusion based reparameterization
(center),  and zoomed image (right).}
\label{fig:flattening2}
\end{figure}


%
%
\subsection{Main Result}  \label{mainresult}
Let $R \subset \mathbb{R}^d$ be an open bounded connected domain of unit measure
with a piecewise smooth boundary $\partial R$.  Suppose that $\varphi: R
\rightarrow \Omega$ is a diffeomorphism, and suppose that $J(x) = \frac{\partial
\varphi_i}{\partial x_j}(x)$ is the Jacobian matrix of this diffeomorphism.  Let
$0 = \eta_0 < \eta_1 \le \eta_2 \le \cdots$ and $v_0,v_1,v_2,\ldots$ denote the
Neumann Laplacian eigenvalues and eigenfunctions on $R$, respectively, and let
$0 = \mu_0 < \mu_1 \le \mu_2 \le \cdots$ and $u_0,u_1,u_2,\ldots,$ denote the
Neumann Laplacian eigenvalues and eigenfunctions on $\Omega$, respectively. We
refer to $R$ as the reference domain, and $\Omega$ as the deformed domain. For
example, the domain $R$ may be some nice domain, where the Neumann Laplacian
eigenfunctions are well understood, while $\Omega$ is a deformed version of this
nice domain.  To study the stability of the first Neumann Laplacian
eigenfunction we will consider the quantity 
$$
\|(\Id - P_{V_k}) u_1 \circ \varphi \|_{L^2(R)}, \quad \text{for} \quad k
=1,2,\ldots,
$$
where $\Id - P_{V_k}$ is the projection on the space orthogonal to $V_k =
\Span\{ v_1,\ldots,v_k\}$. This quantity measures how much
of the function $u_1 \circ \varphi$ cannot be represented in the linear span of
the first $k$ Neumann Laplacian eigenfunctions $v_1,\ldots,v_k$. In Remark
\ref{ex1}, we describe how bounds on the quantity $\|(\Id - P_{V_k}) u_1
\circ \varphi \|_{L^2(R)}$ can be used to understand the observed stability of
the first Neumann Laplacian eigenfunction on tall thin domains.

\begin{theorem} \label{thm1}
Suppose that the singular values of the Jacobian matrix $J(x)$ are contained in
the interval $(1-\varepsilon,1+\varepsilon)$ on $R$, and suppose
that $0 \le \varepsilon  d \le 1/10$. Then,
$$
\left\| \left(\Id - P_{V_k} \right) u_1 \circ \varphi \right\|_{L^2(R)}^2 \le 20
 \frac{\eta_1  \varepsilon   d }{\eta_{k+1} - \eta_1} +
\varepsilon  d,
$$
where $\Id - P_{V_k}$ denotes the projection on the space orthogonal to $V_k
:= \Span \{ v_1,\ldots,v_k\}$.
\end{theorem}

Informally speaking, Theorem \ref{thm1} says that if the deformation $\varphi$
is small and the gap $\eta_{k+1} - \eta_1$ is large, then the function $u_1
\circ \varphi$ is roughly contained in the linear span of the first $k$ Neumann
Laplacian eigenfunctions $v_1,\ldots,v_k$.  We note that the constants appearing
in the statement of Theorem \ref{thm1} have not been optimized, and are for
illustrative purposes.  Furthermore, we remark that the term $\varepsilon 
d$ that appears in the right hand side of the above inequality can be removed if
instead of $\|(\Id - P_{V_k}) u_1 \circ \varphi \|_{L^2(R)}$, we consider the
quantity $\|(\Id - P_{\overline{V}_k}) u_1 \circ \varphi \|_{L^2(R)}$, where
$\overline{V}_k = \Span\{v_0,v_1,\ldots,v_k\}$.  In the following remark, we
apply Theorem \ref{thm1} to understand the case where the reference domain is a
tall rectangle.

\begin{remark}[Tall thin reference domain] \label{ex1}
Suppose that the rectangle $R = [0,1/10] \times [0,10] \subset
\mathbb{R}^2$ is given.  For this rectangle, the first $100$ Neumann Laplacian
eigenfunctions $v_1,\ldots,v_{100}$ are functions of only the $y$-variable
(which varies in the vertical direction).  In particular, these eigenfunctions
are of the form $\cos( \pi j y /10)$, for $j = 1,\ldots,100$.  If we include the
constant eigenfunction $v_0$, and define $\overline{V}_{100} = \Span
\{v_0,v_1,\ldots,v_{100}\}$, then by the discussion following Theorem
\ref{thm1} we have,
$$
\left\| \left( \Id - P_{\overline{V}_{100}} \right) u_1 \circ \varphi
\right\|_{L^2(R)}^2 \le 20  \frac{2 (\pi^2/10^2) \varepsilon}{\pi^2 10^2 - \pi^2}
< \frac{\varepsilon}{200},
$$
when $\varepsilon > 0$ is sufficiently small. Moreover, since the first $100$
Neumann Laplacian eigenfunctions are functions of the $y$-variable, the above
inequality has an interesting consequence. Let $V$ denote the span of all
functions on the rectangle that only depend on the $y$-variable, and set $P_H =
\Id - P_V$. Then,
$$
\left\| P_H u_1 \circ \varphi \right \|_{L^2(R))} < \frac{\varepsilon}{200}.
$$
That is to say, the function $u_1 \circ \varphi$ is essentially a function of
the $y$-variable of the underlying tall thin rectangle, and it has very little
variation in the horizontal direction. As a consequence, level sets of the
function $u_1 \circ \varphi$ are approximately horizontal lines.
\end{remark}

\subsection{Diffusion geometry methods}

Our approach falls under the class of diffusion geometry machine learning
methods, which are popular methods of organizing data that have a theoretical
basis in analysis. Diffusion geometry methods originated with Diffusion Maps
\cite{Lafon:2004}, and are based on constructing a diffusion operator on data
whose eigenfunctions encode the pertinent geometric information for the given
application, see for example \cite{Lederman:2014, Hirn:2014, Talmon:2013b,
Singer:2011, Schclar:2010, Plessis:2009}.  In this paper we present a
specialized construction of a diffusion operator whose first eigenfunction can
be used to organize the layers within a seismic image. To justify the robustness
of the proposed method, we study the stability of these eigenfunction under
domain deformations. Our main result, Theorem \ref{thm1}, acts as a first step
to understanding the stability of the algorithm. Diffusion geometry methods are
highly related to other spectral methods such as Laplacian eigenmaps
\cite{Belkin:2003b}, and Theorem \ref{thm1} also serves to partially illuminate
such techniques.

\subsection{Manifold Straightening}
Let us reiterate the main idea of the application using the language of manifold
learning.  Suppose that $\mathcal{M}$ is a manifold whose metric is given
locally, but for which no global coordinates are given.  Furthermore, suppose
that the manifold has a smooth vector field, which provides an orientation to
each neighborhood.  Assume that the given manifold is a distorted version of an
underlying manifold where the vector field is constant, i.e, all vectors point
in the same direction.  Our objective is to recover the coordinates for this
underlying intrinsic manifold. Our approach has three basic steps:

\begin{enumerate}
\item Contract the metric in the direction of the vector field.
\item Compute the first Neumann Laplacian eigenfunction on the
stretched domain.
\item Use this Neumann Laplacian eigenfunction to recover the height variable on the
underlying manifold.
\end{enumerate}

\begin{figure}[ht!]
\centering
\includegraphics[width=.6\textwidth]{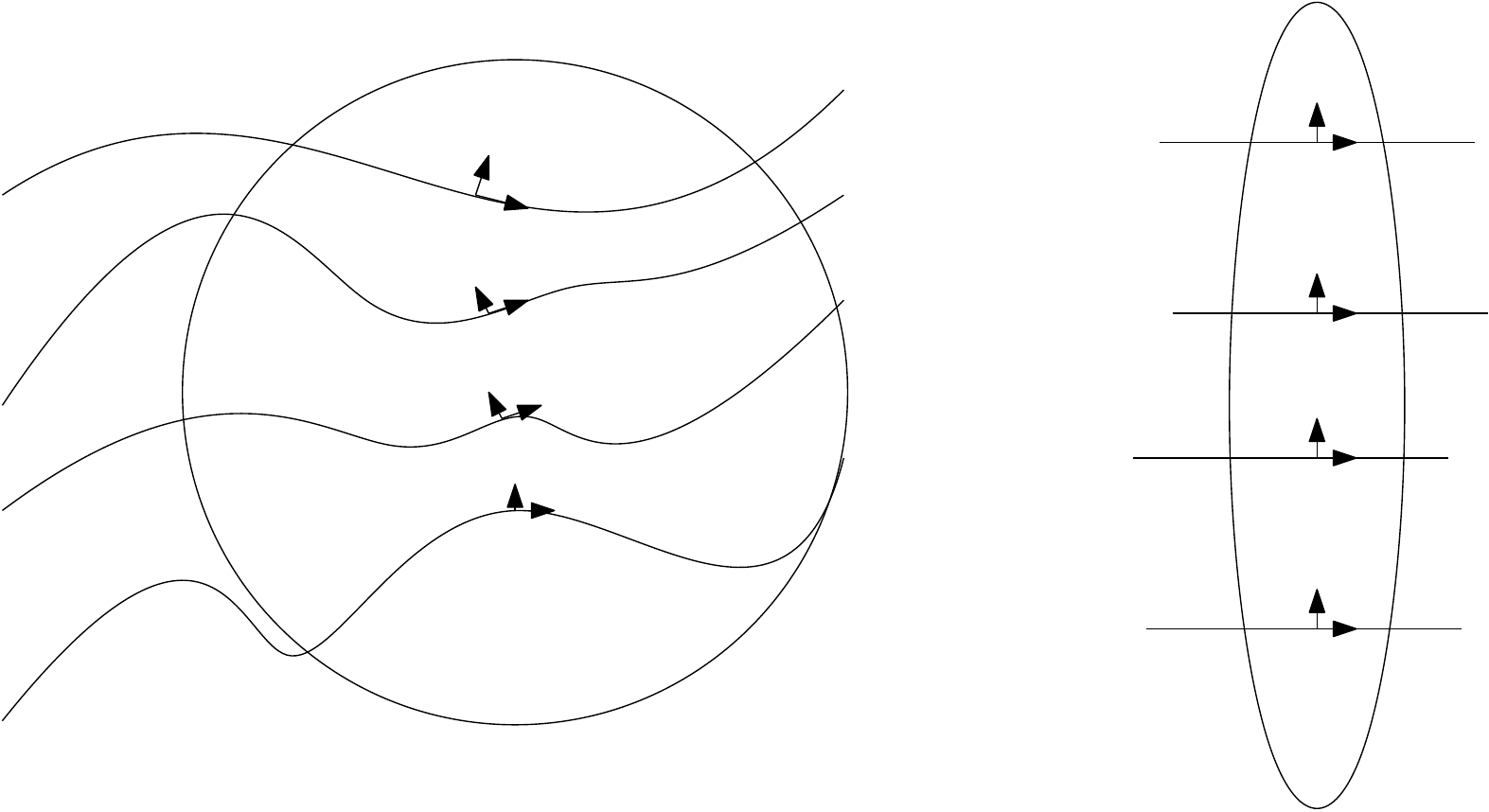}
\caption{Cartoon of a manifold with a vector field (left), and
result of contracting the metric along the vector field (right).} \label{fig:manifold}
\end{figure}
In particular, when the underlying domain is approximately shaped like a tall
rectangle, the height variable can be recovered from the eigenfunction using the
$\arccos$ function.

\subsection{Organization} The remainder of the paper is organized as follows. In
\S\ref{sec:proof} we introduce notation, outline the main idea of the
proof of the Theorem, and provide a detailed proof. In \S\ref{sec:application}
we introduce the application problem, outline the main idea of the algorithm,
and display results. In \ref{sec:details}, the details of the data
adaptive smoothing method, diffusion operator construction, and
reparameterization method are provided.
\section{Proof of main result} \label{sec:proof}

\subsection{Notation and preliminaries} Suppose that $R$ is an open bounded
connected subset of $\mathbb{R}^d$ of unit measure with a piecewise smooth
boundary $\partial R$. We say that $0 = \eta_0 < \eta_1 \le \eta_2 \le \cdots$
and $v_0,v_1,v_2,\ldots$ and the Neumann Laplacian eigenvalues and
eigenfunctions, respectively, if they are solutions to the boundary value
problem
$$
\left\{ \begin{array}{cc}
- \Delta v_j = \eta_j v_j & \text{in } R \\
\frac{\partial}{\partial n} v_j = 0 & \text{on } \partial R,
\end{array} \right.
$$
where $n$ denotes an outward normal to the boundary of the domain. Suppose that
$\varphi : R \rightarrow \Omega$ is a diffeomorphism, and let $0 = \mu_0 < \mu_1
\le \mu_2 \le \cdots$ and $u_0,u_1,u_2,\ldots$ denote the Neumann Laplacian
eigenvalues and eigenfunctions on $\Omega$, respectively. Let $J(x) =
\frac{\partial \varphi_i}{\partial x_j}(x)$ denote the Jacobian matrix of
the diffeomorphism $\varphi$. If $y := \varphi(x)$ and $u \in C^1(\Omega)$, 
then the chain rule can be expressed by
\[
J^T(\varphi^{-1}(y)) \nabla_y u(y) = \nabla_x u(\varphi(x)),
\]
where 
$$
\nabla_y = \left( \frac{\partial}{\partial y_1}, \frac{\partial}{\partial
y_2} \right)^\top \quad \text{and} \quad \nabla_x =
\left( \frac{\partial}{\partial x_1}, \frac{\partial}{\partial x_2}\right)^\top.
$$
Similarly, given a function $v \in C^1(R)$, the chain rule can be expressed by
\[ 
(J^{-1})^T(\varphi^{-1}(x)) \nabla_x v(x) = \nabla_y v(\varphi(y)).
\]
Let $\det J(x)$ denote the Jacobian determinant such that 
$$
\int_\Omega u(y) dy = \int_R u(\varphi(x)) |\det J(x)| dx ,
\quad \text{and} \quad
\int_R v(x) dx = \int_\Omega v(\varphi^{-1}(y)) | \det J^{-1} (y)| dy.
$$
The assumption that the singular values of $J(x)$ are contained in the interval
$[1-\varepsilon,1+\varepsilon]$ on $R$ implies that
\[
(1-\varepsilon)^d \le \det J(x) \le (1+\varepsilon)^d.
\]
Recall the following two classical inequalities involving Euler's number $e$:
$$
\left(1 + \frac{1}{x}\right)^x < e
\quad \text{and} \quad
e^{-1} <
\left(1 - \frac{1}{x} \right)^{x-1},
\quad \text{for all} \quad 
1 < x < \infty.
$$
It follows directly from Taylor's Inequality that
$$
|e^x - 1| < e^{1/10} x \quad \text{when} \quad |x| < \frac{1}{10}.
$$
Therefore, assuming that $\varepsilon d < 1/10$, it follows that
\[
(1+\varepsilon)^d = (1+\varepsilon)^\frac{\varepsilon d}{\varepsilon}
\le e^{\varepsilon d}  < 1 + e^{1/10} \varepsilon d < 1+ 2 \varepsilon d,
\]
and similarly that,
$$
(1-\varepsilon)^d = (1-\varepsilon)^{\left( \frac{1}{\varepsilon}-1 \right) 
\left(\frac{\varepsilon}{1-\varepsilon}\right) d}
> e^{-\frac{\varepsilon}{1-\varepsilon}d} > 1- e^{1/10}
\frac{\varepsilon}{1-\varepsilon}  d > 1 - 2 \varepsilon d.
$$
By combining these inequalities we conclude that  if the singular values of
$J(x)$ are contained in the interval $[1-\varepsilon,1+\varepsilon]$ on $R$, and
$0 \le \varepsilon d < 1/10$, then 
$$
1 - 2 \varepsilon d \le \det J(x) \le 1 + 2 \varepsilon d.
$$
In the following, while establishing the lemmas involved in the proof of Theorem
\ref{thm1}, we will assume that the singular values of $J(x)$ are contained in
the interval $[1-\varepsilon,1+\varepsilon]$, and will assume that
$$
1 - \delta \le \det J(x) \le 1 + \delta.
$$
Then, to complete the proof of Theorem \ref{thm1}, we will substitute $\delta =
2 \varepsilon d$. In addition to simplifying notation, controlling the
determinant and singular values of $J(x)$ separately is useful for understanding
certain classes of transformations; for example, the shear transform 
$$
(x_1,x_2) \overset{\varphi}{\mapsto} (x_1+ \alpha x_2,x_2),
$$
has arbitrarily large singular values as $\alpha \rightarrow \infty$, but has
Jacobian matrix
\[
J(x) = \left( \begin{array}{cc} 
1 & \alpha \\
0 & 1
\end{array} \right),
\]
which has determinant $1$.

\subsection{Outline of the proof of Theorem \ref{thm1}}
Let $\gtrsim$
denote $\ge$ up to multiplication by a constant equal to $1 + O(\varepsilon)$.
Recall that $\eta_j$ and $v_j$ are the Neumann Laplacian eigenvalues and
eigenfunctions on $R$, while $\mu_j$ and $u_j$ are the Neumann Laplacian
eigenvalues and eigenfunctions on $\Omega$. Observe that the condition that the
singular values of $J(x)$ are contained in the interval
$[1-\varepsilon,1+\varepsilon]$ on $R$ can be succinctly stated as
$$
\sigma \left(JJ^\top \right) \subset
\left[(1-\varepsilon)^2,(1+\varepsilon)^2 \right].
$$
The proof of Theorem \ref{thm1} is divided into three lemmas.  First, in Lemma
\ref{lem1} we will show 
$$
\eta_1 = \|\nabla_x v_1\|^2_{L^2(R)} \gtrsim \|\nabla_y( v_1 \circ \varphi^{-1}
)\|_{L^2(\Omega)}^2 \gtrsim \inf_{u \perp 1 : \|u\|=1} \|\nabla_y u\|_{L^2(\Omega)}^2 =
\mu_1,
$$
where the infimum is taken over sufficiently smooth functions of 
unit $L^2(\Omega)$-norm, which are orthogonal to constant functions on
$L^2(\Omega)$.  Second, in Lemma \ref{lem2} we will show that 
\[
\mu_1 = \|\nabla_y u_1\|_{L^2(\Omega)}^2 \gtrsim \|\nabla_x (u_1 \circ
\varphi)\|_{L^2(R)}^2 = \sum_{j=1}^\infty \eta_j \alpha_j^2,
\]
where $\alpha_j$ are the coefficients of $u_1 \circ \varphi$ expanded in the
orthogonal basis $\{v_j\}$ of Neumann eigenfunctions on $R$. Third, in Lemma
\ref{lem3} we show that  $\alpha_0^2$ is $\mathcal{O}(\varepsilon^2)$.
Finally, in \S \ref{finish} we combine these lemmas to conclude that
\[
\eta_1 \gtrsim \sum_{j=1}^\infty \eta_j \alpha_j^2,
\]
and then we use this inequality to control how large the coefficients
$\alpha_j^2$ can be. In particular, since $\eta_j$ is increasing, observe that
this inequality provides increasing control as $j$ increases.

\subsection{Lemmas involved in the proof of Theorem \ref{thm1}} \label{lemmas}

\begin{lemma} \label{lem1}
Suppose that $\sigma \left(JJ^\top \right) \subset
\left[(1-\varepsilon)^2,(1+\varepsilon)^2 \right]$ and $|\det J(x) - 1| <
\delta$ on $R$. Then,
\[
\eta_1 \frac{(1+\varepsilon)^2 (1-\delta)}{(1-\delta)^3 - (1+\delta) \delta^2}
\ge \mu_1.
\]
\end{lemma}
\begin{proof}
Multiplying
the equation $\Delta v_1 + \eta_1 v_1 = 0$ by $v_1$, integrating over $R$, and
applying Green's Identity yields
\[
\eta_1 = \| \nabla_x v_1 \|_{L^2(R)}^2,
\]
which by the chain rule, and a change of variables of integration is equivalent
to
\[
\eta_1 = \| J^T \circ \varphi^{-1} \nabla_y \circ \varphi^{-1} \sqrt{|\det
J^{-1}|} \|_{L^2(\Omega)}^2. 
\]
Therefore, by the assumed bounds for $\sigma_\text{max}(J)$ and $\det J$, we
conclude that
\[
\eta_1  \frac{(1+\varepsilon)^2}{1-\delta} \ge \| \nabla_y v_1 \circ
\varphi^{-1} \|_{L^2(\Omega)}^2.
\]
By the minimax principle, we have
\[
\mu_1 = \inf_{u \perp 1} \frac{ \| \nabla_y u
\|^2_{L^2(\Omega)}}{\|u\|_{L^2(\Omega)}^2},
\]
where the infimum is taken over differentiable functions $u$, which are
orthogonal to constant functions on $L^2(\Omega)$. It follows that
\[ 
\frac{\| \nabla (v_1 \circ \varphi^{-1} -c) \|_{L^2(\Omega)}^2}{\|v_1
\circ \varphi^{-1} -c \|_{L^2(\Omega)}^2} \ge \mu_1,
\]
where $c$ is a constant such that the function $v_1 \circ \varphi^{-1} -c $ is
orthogonal to constant functions on $L^2(\Omega)$, i.e.,
\[ 
c = \frac{1}{\|\Omega|} \int_\Omega v_1 \circ \varphi^{-1} dx .
\]
Now, combining this inequality with the previous inequality comparing
$\eta_1$ and $\| \nabla_y v_1 \circ \varphi^{-1}\|_{L^2(\Omega)}^2$ yields 
\[
\eta_1 \frac{(1 + \epsilon)^2}{(1 - \delta)} \frac{1}{\|v_1 \circ \varphi^{-1} -
c \|_{L^2(\Omega)^2}^2} \ge 
\frac{ \|\nabla_y (v_1 \circ \varphi^{-1} - c) \|_{L^2(\Omega)}^2}{\|v_1 \circ
\varphi^{-1} - c \|_{L^2(\Omega)}^2} \ge \mu_1.
\]
To complete the proof, it remains to compute a lower bound on $\| v_1 \circ
\varphi^{-1} - c\|_{L^2(\Omega)}^2$. The constant $c$ can be computed by taking
the inner product of $v_1 \circ \varphi^{-1}$ with the constant function
$|\Omega|^{-1}$, specifically,
\begin{eqnarray*}
c = |\Omega|^{-1} \int_{\Omega} v_1 ( \varphi^{-1}(y)) dy 
&=& |\Omega|^{-1} \int_R v_1(x) | \det J(x)| dx  \\
&\le&  \delta |\Omega|^{-1} \int_R |v_1(x)|dx 
\le \delta |\Omega|^{-1} |R|
\le \frac{\delta}{1-\delta}.
\end{eqnarray*}
Since $v_1 \circ \varphi^{-1} - c$ is orthogonal to constant functions in
$L^2(\Omega)$, 
\begin{eqnarray*}
\|v_1 \circ \varphi^{-1} - c\|^2_{L^2(\Omega)} 
&=& \| v_1 \circ \varphi^{-1} \|_{L^2(\Omega)} - \|c\|^2_{L^(\Omega)} \\
&\ge& \int_R |v_1(x)|^2 |\det J(x)| dx - |\Omega| \left( \frac{\delta}{1-\delta}
\right)^2  \\
&\ge& (1-\delta) - (1+\delta) \frac{\delta^2}{(1-\delta)^2} \\
&=& \frac{(1-\delta)^3 - (1+\delta) \delta^2}{(1-\delta)^2}.
\end{eqnarray*}
Substituting this lower bound into our previous inequality yields:
\[
\eta_1 \frac{(1+\varepsilon)^2 (1-\delta)}{(1-\delta)^3 - (1+\delta) \delta^2}
\ge \mu_1.
\]
\end{proof}
\begin{lemma} \label{lem2}
Suppose that $\sigma \left(JJ^\top \right) \subset
\left[(1-\varepsilon)^2,(1+\varepsilon)^2 \right]$ and $|\det J(x) - 1| <
\delta$ on $R$. Then,
$$
\mu_1 \frac{(1+\delta)}{(1-\varepsilon)^2} \ge \sum_{j=1}^\infty \eta_j
\alpha_j^2,
$$
where $\alpha_j$ are the coefficients of $u_1 \circ \varphi$ expanded in the
orthogonal basis of eigenfunctions $\{v_j\}$ on $L^2(R)$.
\end{lemma}
\begin{proof}
First, we express $\mu_1$ in terms of an integral of the function $\mu_1 \circ
\varphi$ over $R$,
$$
\mu_1 = \left\| (J^{-1})^T \circ \varphi \nabla_x (u \circ \varphi) \sqrt{|\det
J|} \right\|_{L^2(R)}^2.
$$
This construction is symmetric to our equation for $\eta_1$ in terms of an
integral of $v_1 \circ \varphi^{-1}$ over $\Omega$. Applying the bounds
on $\sigma_\text{min}(J)$ and $\det J$ yields
\[
\mu_1 \frac{(1+\delta)}{(1-\varepsilon)^2} \ge \left\|\nabla_x (u_1 \circ
\varphi) \right\|_{L^2(R)}^2.
\]
Since the Neumann Laplacian eigenfunctions $\{v_j\}$ form a orthonormal basis of
$L^2(R)$, we can expand 
\[
u_1 \circ \varphi(x) = \sum_{j=0}^\infty \alpha_j v_j(x),
\]
where
\[
\alpha_j = \int_R u_1(\varphi(x)) v_j(x) dx.
\]
By Green's first identity, and the orthogonality of the eigenfunctions
$\{v_j\}$, we have
\[
\left\| \nabla_x \sum_{j=1}^\infty \alpha_j v_j \right\|_{L^2(R)}^2 =
\sum_{j=1}^\infty \eta_j \alpha_j^2.
\]
Combining the previously developed inequalities and expanding $u_1 \circ
\varphi$ in $\{v_j\}$ establishes the inequality
\[
\mu_1 \frac{(1+\delta)}{(1-\varepsilon)^2} \ge \left\|\nabla_x (u_1 \circ
\varphi) \right\|_{L^2(R)}^2
 \ge \sum_{j=1}^\infty \eta_j
\alpha_j^2,
\]
where $\alpha_j$ are the coefficients of $u_1 \circ \varphi$ expanded in the
orthogonal basis of eigenfunctions $\{v_j\}$ on $L^2(R)$.
\end{proof}
Observe that the sum in the inequality that we establish excludes the term
$\eta_0 \alpha_0^2$, as $\eta_0 = 0$.  Therefore, our inequality offers no
control over $\alpha_0^2$. However, since $u_1$ is orthogonal to constants on
$L^2(\Omega)$, we suspect $u_1 \circ \varphi$ is nearly orthogonal to constants
on $L^2(R)$. In the following lemma, we formalize this intuition, and develop
a precise bound for $\alpha_0^2$.

\begin{lemma} \label{lem3}
Suppose that $|\det J(x) - 1| < \delta$ on $R$. Then,
\[
\alpha_0^2 \le \frac{\delta^2 (1+\delta)^2}{(1-\delta)^2}.
\]
\end{lemma}
\begin{proof}
By definition,
$$
\alpha_0^2 = \left( \int_R u_1(\varphi(x)) dx \right)^2
= \left( \int_\Omega u_1(y) |\det J^{-1}(y)| dy \right)^2.
$$
Moreover, we have
$$
\left( \int_\Omega u_1(y) |\det J^{-1}(y)| dy \right)^2 
\le \left( \frac{\delta}{1 - \delta} |\Omega| \right)^2
\le \frac{\delta^2 (1+\delta)^2}{(1-\delta)^2},
$$
which completes the proof.
\end{proof}

\subsection{Proof of Theorem \ref{thm1}} \label{finish}
In the following, we combine Lemmas \ref{lem1}, \ref{lem2}, and \ref{lem3} to
prove Theorem \ref{thm1}.

\begin{proof}[Proof of Theorem \ref{thm1}]
By combining the results of Lemma \ref{lem1} and \ref{lem2}, we conclude that
\[
\eta_1 E_{\delta,\varepsilon} \ge \sum_{j=1}^\infty \eta_j \alpha_j^2,
\]
where 
\[
E_{\delta,\varepsilon} = \frac{(1+\varepsilon)^2}{(1-\varepsilon)^2}
\frac{(1+\delta)(1-\delta)}{(1-\delta)^3 - (1+\delta) \delta^2}.
\]
Subtracting $\sum_{j=1}^{k} \eta_j \alpha_j^2$ from each side of this inequality
yields
\[
\eta_1 E_{\varepsilon,\delta}  - \sum_{j=1}^{k} \eta_j \alpha_j^2 \ge
\sum_{j={k+1}}^\infty \eta_j \alpha_j^2 .
\]
Then, since the eigenvalues $\eta_1 \le \eta_2 \le \eta_3 \le \cdots$ are in
ascending order, we conclude that
$$
\eta_1 \left(E_{\varepsilon,\delta} - \sum_{j=1}^{k} \alpha_j^2 \right) \ge
\eta_{k+1} \sum_{j=k+1}^\infty \alpha_j^2.
$$
By the inequality $1/(1+\delta) \le \|u_1 \circ \varphi\|_{L^2(R)} =
\sum_{j=0}^\infty \alpha_j^2$, it follows that
\[
\eta_1 \left( E_{\varepsilon,\delta}  - \left(\frac{1}{1+\delta} -
\sum_{j=k}^\infty \alpha_j^2 - \alpha_0 \right) \right) \ge \eta_{k+1}
\sum_{j=k+1}^\infty \alpha_j^2.
\]
Rearranging  and collecting terms gives
\[
\frac{\eta_1 F_{\varepsilon,\delta}}{\eta_{k+1} - \eta_1} \ge
\sum_{j=k+1}^\infty \alpha_j^2,
\]
where
\[
F_{\varepsilon,\delta} = \frac{(1+\varepsilon)^2}{(1-\varepsilon)^2} 
\frac{(1+\delta)(1-\delta)}{(1-\delta)^3 + (1+\delta) \delta^2} +
\frac{\delta^2(1+\delta)^2}{(1-\delta)^2} - \frac{1}{1+\delta}.
\]
Let $\delta = 2 \varepsilon d$. Using Lemma \ref{lem3}, power series expansions,
and assuming $\varepsilon d < 1/10$ gives
\[
20  \frac{\eta_1  \varepsilon  d}{\eta_{k+1} - \eta_1} 
+ \varepsilon d
\ge
\alpha_0^2 + \sum_{j=k+1}^\infty \alpha_j^2 ,
\]
as was to be shown.
\end{proof}
\section{Applications} \label{sec:application}
\subsection{Problem} In the field of geophysical interpretation, the problem of
organizing the layers of a seismic image, otherwise known as seismic flattening,
is of interest \cite{Lomask:2007,Gao:2009,Parks:2010}. Given a
seismic image, the objective is to modify the depth function of the seismic
image such that each layer has constant depth.  Since each layer was deposited
at roughly the same time, the depth function which flattens the layers is
sometimes referred to as geological time. In general, a seismic image consists
of an $m \times n \times l$ real-valued tensor whose first coordinate is
referred to as depth.  As the depth increases, the values in the tensor
oscillate between positive and negative values, but in a nonuniform way.  For
example in Figure \ref{fig:flattening} a two dimensional slice of a
three dimensional seismic image is plotted using a color map, which exhibits
this oscillating behavior. 
\begin{figure}[h!]
\centering
\includegraphics[width=0.47\textwidth]{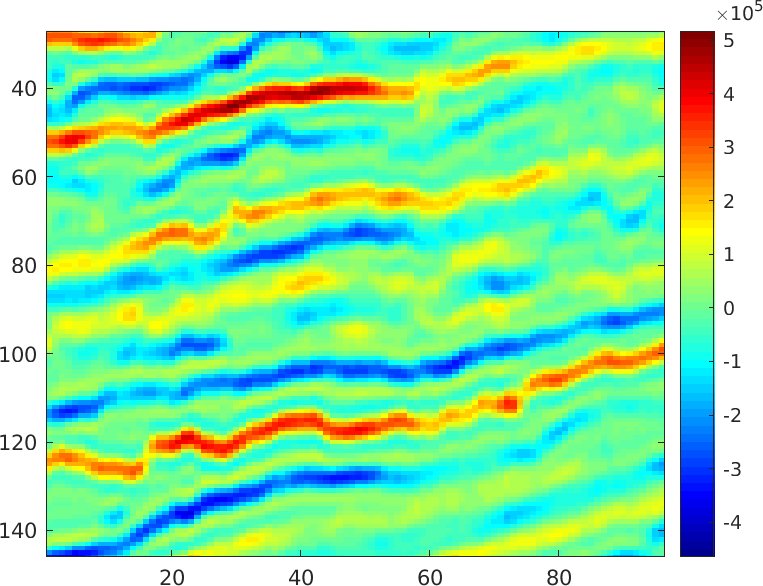}
\caption{Example of a two dimensional slice of a seismic image tensor.}
\label{fig:flattening}
\end{figure}

In the following, we will primarily restrict our
attention to such two dimensional slices of seismic tensors, except 
for the filtering process on the data, which does incorporate the pixels
surrounding our two dimensional slice in the tensor.

\subsection{Approach Outline} We approach the problem of seismic flattening from
a diffusion geometry perspective. Given a seismic image, we construct a
diffusion process whose eigenfunctions encode the geometric features of the data
that are of interest, which, in this case, are the layers of the seismic image.
Our method has three main steps:
\begin{enumerate}
\item \textbf{Adaptive Filtering.} A Principal Component Analysis filtering
technique is used to associate filtered feature vectors to each pixel in the
two dimensional seismic image. These filtered features incorporate
information from the surrounding pixels. 
\item \textbf{Kernel Construction.} A diffusion kernel is defined, which
propagates rapidly along the layers of the seismic image and
slowly perpendicular to these layers. This kernel is defined using the filtered
feature vectors.
\item \textbf{Layer Organization.} The first eigenfunction of the diffusion
operator (which should resemble the first Neumann Laplacian eigenfunction on the
intrinsic underlying domain that is roughly shaped like a tall thin rectangle)
is then used to organize the layers of the image.
\end{enumerate}

Intuitively, since the constructed diffusion process propagates rapidly along
the layers of the seismic image, and slowly perpendicular to the layers, if the
diffusion operator is taken to a sufficient power, diffusion starting at
a single point on a layer will completely propagate along the layer, with very
little propagation perpendicular to the layer. This powered diffusion operator
essentially replaces functions with their average along the layers of the
seismic image. However, the eigenfunctions of the original diffusion operator
and the powered diffusion operator are the same. Therefore, the first nontrivial
eigenfunction of the diffusion operator must be essentially equal to its
average on the layers of the image. That is to say, the first eigenfunction must
be essentially constant on the layers of the image. A similar idea was used by
Lafon, see Proposition 12 on page 47 of \cite{Lafon:2004}, where an anisotropic
diffusion process is used to study differential systems.

\subsection{Main idea}
 A complete description of the proposed method is provided in \ref{sec:details}.
Here, we focus on explaining the kernel construction, which is the main idea of
the proposed method. Suppose that the filtered feature vectors $f(i)$, which
summarizes the structure of the seismic image at $i$, have already been computed
using the method described in \ref{sec:details}. For each pixel $i$ we consider
two spatial neighborhoods: a calibration neighborhood $C_R$, and a propagation
neighborhood $N_r$, where $0 < r < R$ denotes the radius of each neighborhood.

\begin{figure}[h!]
\centering
\includegraphics[width=.3\textwidth]{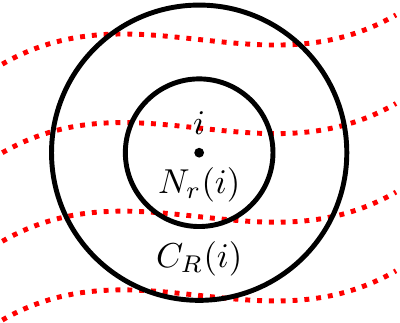}
\caption{Illustration of propagation and calibration neighborhood.}
\label{fig06} 
\end{figure}
The calibration neighborhood $C_R(i)$ determines the level of variation of the
filtered feature vectors around pixel $i$. Specifically, we define:
\[
M(i) = \max_{j \in C_R(i)} \|f(i) - f(j)\|_2^2 .
\]
We will use this function $M(i)$ to account for the differing levels of
variation across the seismic image. The propagation neighborhood $N_r(i)$
determines the support of the diffusion operator. Specifically, we define the
kernel
$$
W(i,j) = 
\left\{
\begin{array}{cl}
\exp \left( - \frac{\|f(i) - f(j)\|_2^2}{\varepsilon M(i)} \right)
&
\text{if}  \quad j \in N_r(i)  \\
0 & \text{if} \quad j \not \in N_r(i), 
\end{array} \right.
$$
where $\varepsilon > 0$ is a parameter.
Restricting the propagation of the diffusion process to the small neighborhood
$N_r$ ensures the resulting kernel will be sparse, and hence computationally
efficient to work with.  By appropriately choosing the parameter $\varepsilon >
0$, we can control how much the truncation to $N_r(i)$ effects the diffusion
operator.  
The filtered features are designed in such a way that they are
relatively constant when a pixel is translated along a layer, and differ
greatly
as a pixel moves perpendicular to a layer (The pixel values themselves also have
this property, but with more noise).  Therefore, by choosing the parameters,
$R$, $r$, and $\varepsilon$ are appropriately, the diffusion process will
strongly prefer to travel along the level lines of the image, and be strongly
penalized for traveling perpendicular to these level lines. For this constructed
diffusion kernel, the underlying domain on which the diffusion kernel
corresponds to isotropic diffusion must be very narrow and tall, hence the
connection to Theorem \ref{thm1}. The precise details of the kernel
construction are included in \ref{sec:details}. 

\subsection{Example}  \label{examplesec}
In this section we will illustrate each step of the
described method. 
\subsubsection{Filtered Features} First, each pixel of the selected two
dimensional slice is associated with a filtered feature vector $f(i)$, which
incorporates information
from the pixels surrounding $i$ in the seismic tensor. The filtered feature
vector summarizes the local structure of the seismic image in a smooth
way. In Figure \ref{fig:compare} a two dimensional slice of the seismic
image is plotted verses one of the coordinates of the filtered feature
vector $f(i)$. This data adaptive filtering process serves to remove local noise
and incorporate some of the surrounding
structures.
\begin{figure}[h!]
\centering
\begin{tabular}{cc}
\includegraphics[width=0.47\textwidth]{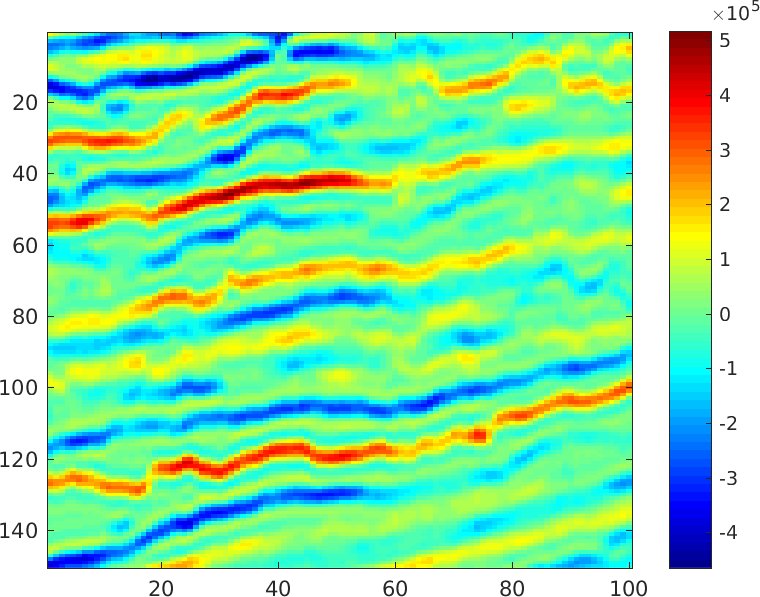} &
\includegraphics[width=0.47\textwidth]{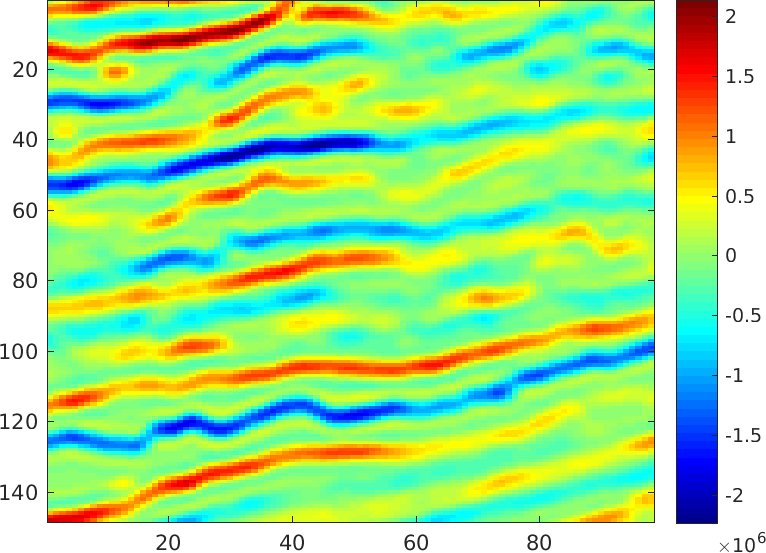}
\end{tabular}
\caption{Seismic image (left), and coordinate of filtered feature vector
(right).}
\label{fig:compare}
\end{figure}

\subsubsection{Diffusion operator eigenfunctions} Next, we define the kernel $W(i,j)$ using the filtered feature vectors and
normalized $W(i,j)$ appropriately to define a diffusion operator $P(i,j)$. The
first two (nontrivial) eigenvectors of $P$ are are plotted in Figure
\ref{fig:eig}. As expected the first eigenvector is approximately
monotone, while the second eigenvector resembles a higher order oscillation. The
subsequent eigenvectors additionally contain oscillations in the horizontal
direction. Therefore, the underlying domain corresponds to some tall rectangular
domain whose height is approximately $2-3$ times its width.
%
%
\begin{figure}[h!]
\centering
\begin{tabular}{cc}
\includegraphics[width=0.47\textwidth]{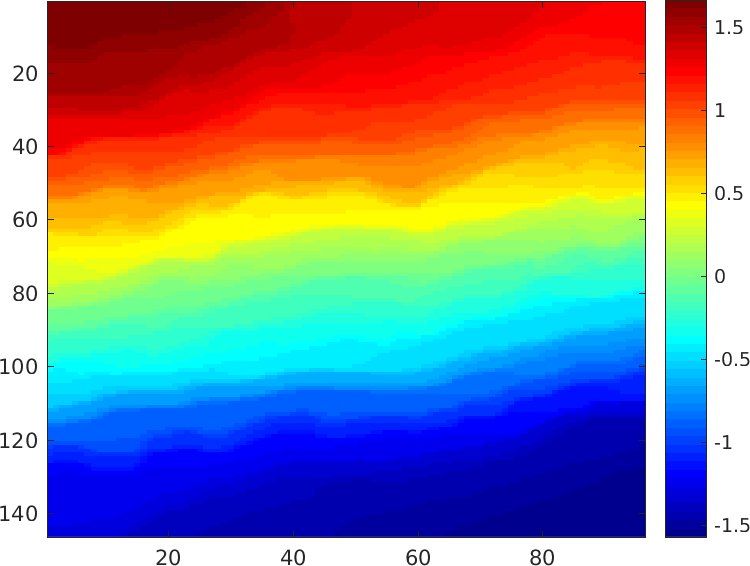} &
\includegraphics[width=0.47\textwidth]{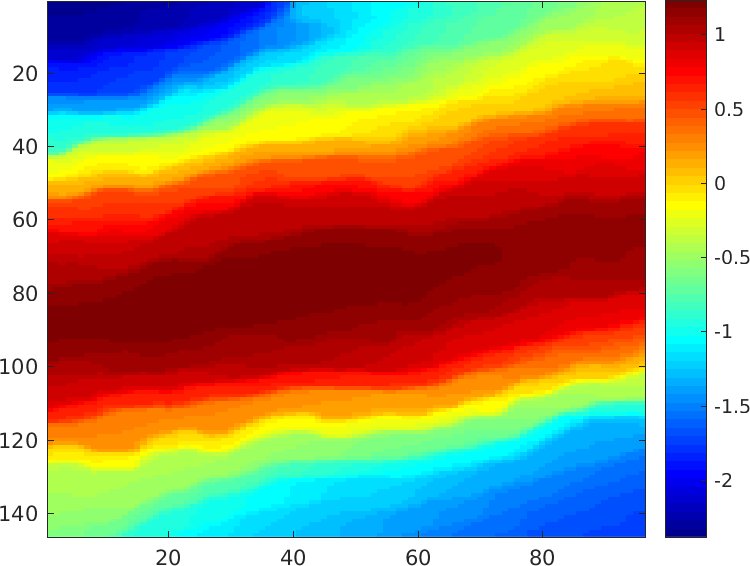}
\end{tabular}
\caption{First (left) and second (right) eigenfunction of the diffusion operator
$P$.}
\label{fig:eig}
\end{figure}
\subsubsection{Flattening} Given the first eigenfunction of the diffusion operator $P$, we can recover
the depth function by assuming this eigenfunction coincides with the first
Neumann Laplacian eigenfunction on some tall thin rectangle. The result of this
flattening procedure is plotted in Figure  \ref{fig:flattening2}.
%
\begin{figure}[h!]
\centering
\begin{tabular}{cc}
\includegraphics[width=0.47\textwidth]{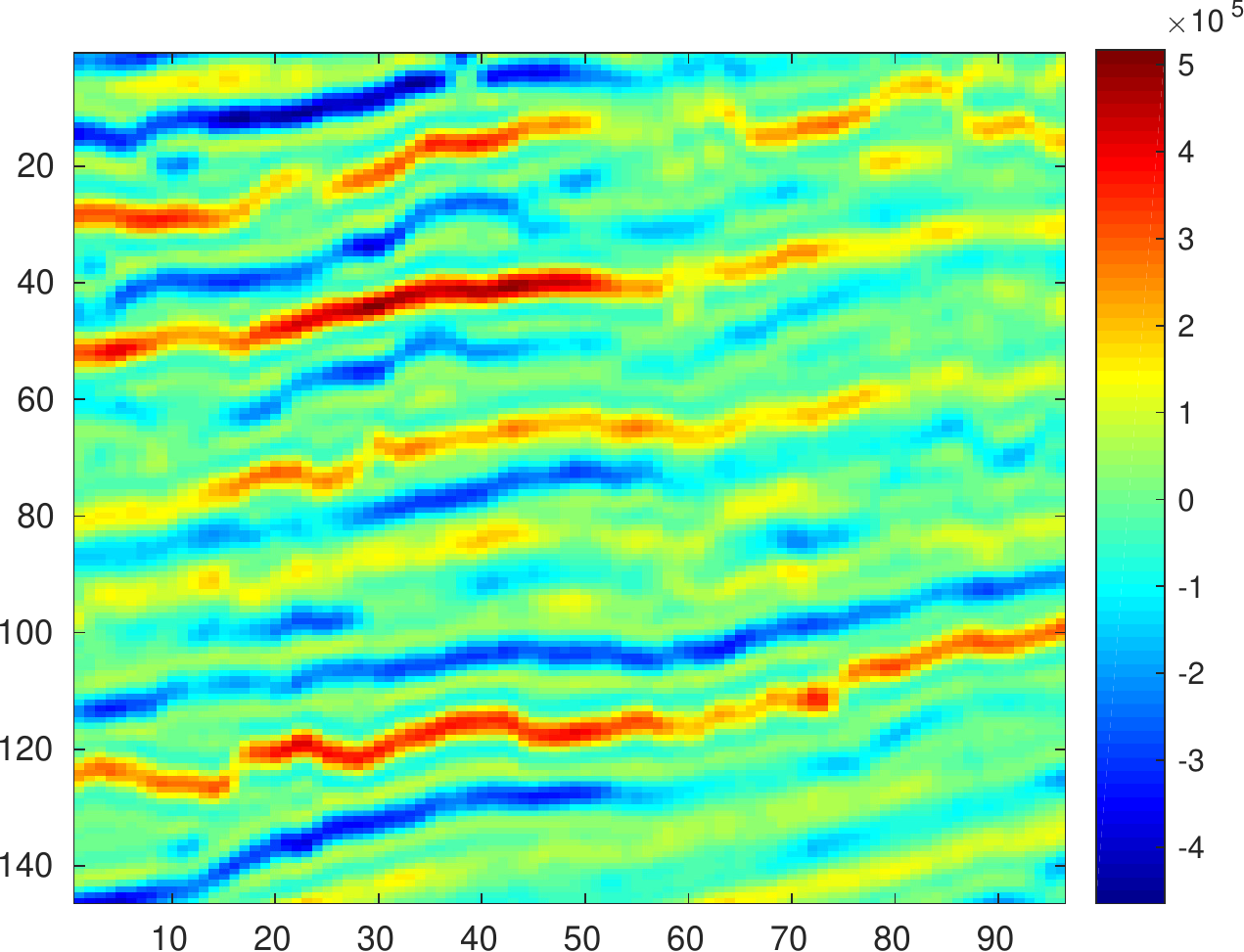} &
\includegraphics[width=0.47\textwidth]{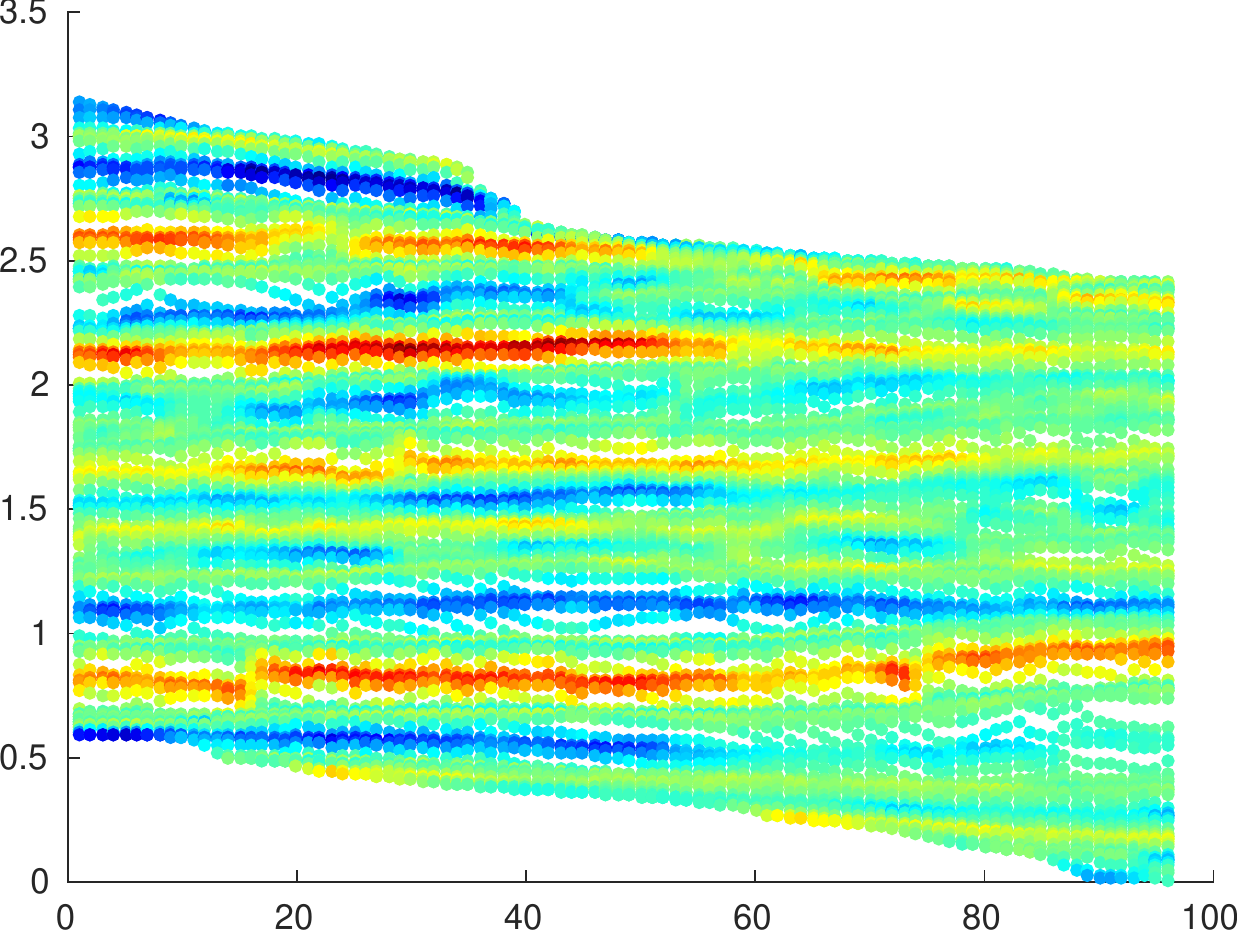} 
\end{tabular}
\caption{Seismic image (left), and result of diffusion flattening (right).}
\label{fig:flattening2}
\end{figure}
\subsubsection{Magnified flattening} Comparing the scatter plot to the original seismic image in Figure
\ref{fig:flattening2} it is evident that the layers are already fairly flat.
However, at this resolution, the markers in the scatter plot are overlapping. By
zooming into the scatter plot and the corresponding area in the seismic image,
we can see that the description generated by the reparameterization is in fact
far richer than simply flattening. In Figure \ref{fig:zoom}, we zoom into the
interval $[90,130]$ on the $y$-axis of the seismic image in  Figure
\ref{fig:flattening2}, and we zoom into a corresponding interval on the $y$-axis
of the reparameterized image.  At this new level of zoom white space appears
between the scatter plot markers.  This white space is simply a result of the
markers diminishing in size, and demonstrates that the diffusion flattening
procedure offers an advantageous side effect: the points within a layer are
automatically clustered when the depth function is reparameterized by the first
eigenfunction of the diffusion operator.
%
%
\begin{figure}[ht!]
\centering
\begin{tabular}{cc}
\includegraphics[width=0.47\textwidth]{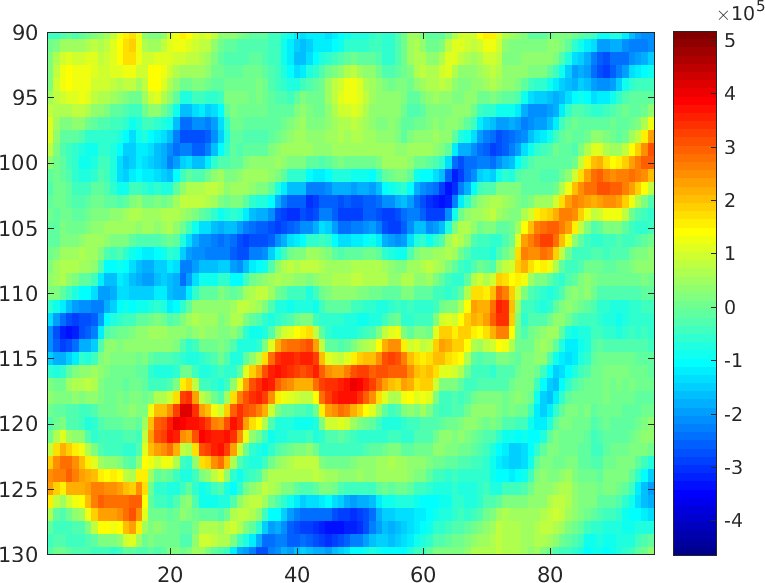} &
\includegraphics[width=0.47\textwidth]{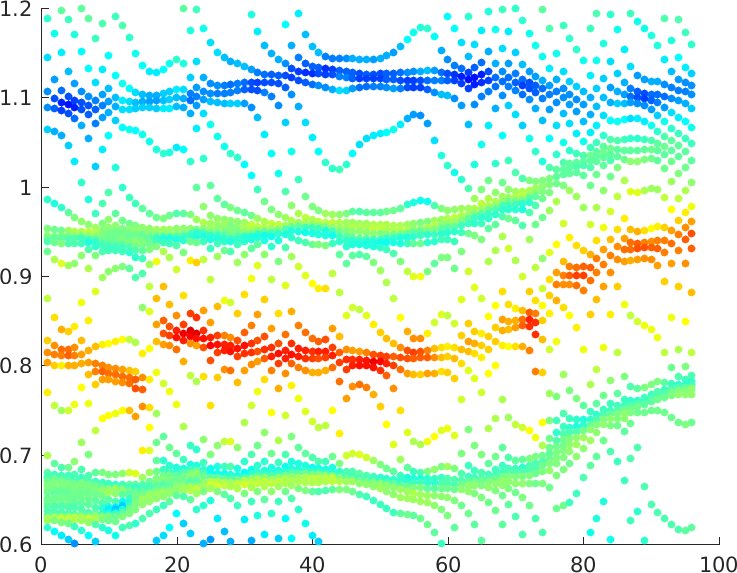}
\end{tabular}
\caption{Seismic image with $y$-axis zoomed into $90$ to $130$, and
corresponding region of flattened image. At this zoom level, the
detailed structure of the diffusion layer organization becomes evident.}
\label{fig:zoom}
\end{figure}
\subsubsection{Example Summary} Given a two dimensional seismic image we
have defined a diffusion process on the image using structural information from
the surrounding pixels. This diffusion processes propagates rapidly
along the layers of the seismic image and slowly perpendicular to these layers.
Therefore, the diffusion corresponds to isotropic diffusion on some
underlying domain, roughly shaped like a tall thin rectangle where the layers of
the seismic image are roughly horizontal.  Therefore, using our intuition backed
by Theorem \ref{thm1}, we can recover the depth function in
the image by assuming the first eigenfunction of the diffusion operator
resembles the first Neumann Laplacian eigenfunction on the rectangle. 

\appendix

\section{Algorithm details} \label{sec:details}

\subsection{Adaptive Filtering}
In this section, we describe the adaptive filtering method. The filtering method
is described for a general three dimensional image since it requires no
additional effort, whereas for the rest of the algorithm we are mainly
interested in the filtered feature vectors computed for a two dimensional slice
of a three dimensional seismic image. Let $X$ denote a three dimensional seismic
image, $i$ denote a pixel from $X$, and $v(i)$ denote the value associated with
pixel $i$.
Each interior pixel $i$ of the seismic image is surrounded by a $3 \times 3
\times 3$ cube of pixels in $X$.  Given an interior pixel $i$, let $g(i)$ denote
the set of values associated with the pixels surrounding $i$. By fixing a method
of rearranging a $3 \times 3 \times 3$ cube of pixels into a vector, we can
consider each $g(i)$ as a vector in $\mathbb{R}^{27}$. Given a collection of feature
vectors $g(i)$, we apply principal component analysis (PCA) to create a low
dimensional description of each $g(i)$. Specifically, let $A$ denote the $27
\times N$ matrix whose columns consist of the feature vectors $g(i)$ for each of
the $N$ interior pixels.  Define the feature covariance matrix 
\[
C = \frac{1}{1-N} \left(A - A \frac{1}{N} 1 1^T \right) 
\left(A - A \frac{1}{N} 1 1^T \right)^T ,
\]
where $1$ denotes a column vector of all ones in $\mathbb{R}^N$. Let $u_1$ denote the
principal component of $C$, that is, the eigenvector of $C$ associated with
the largest eigenvalue. To each pixel $i$, we associate a scalar value $w(i)$
defined by
\[
w(i) = u_1^T g(i).
\]
We refer to $w(i)$ as the filtered pixel value of $i$. Next, we define the
filtered feature vector
$f(i)$ as the vector in $\mathbb{R}^{27}$ whose entries are the filtered pixel
values $w(j)$ of the pixels $j$ in the $3 \times 3 \times 3$ cube of pixels
surrounding pixel $i$. As before, the $3 \times 3 \times 3$ cubes of pixels are
converted into vectors via some fixed method of rearrangement. The filtered
features $f(i)$ are used to define the diffusion process in the following
section.  

\subsubsection*{Summary of Adaptive Filtering} 
To each pixel $i$ we associated the cube of surrounding pixel values
$g(i)$, and used PCA to summarize each cube by a single coordinate $w(i)$.
Finally, we associated each pixel with a filtered feature vector $f(i)$,
corresponding to a cube of the surrounding filtered coordinates $w(i)$.

\subsection{Kernel Construction} 
In this section we describe the construction of a diffusion process on the
pixels of a seismic image, which propagates rapidly along the layers of the
image and slowly perpendicular to these layers. For simplicity we restrict
our attention to a two dimensional slice of a three dimensional seismic image,
as shown in the left side of Figure \ref{fig:compare}.

Let $Y$ denote a two dimensional slice of a given three dimensional seismic
image, $i$ denote a pixel of $Y$, and $f(i)$ denote the filtered feature vector
$f(i) \in \mathbb{R}^{27}$ whose construction is described in the previous section.
Let $x(i)$ denote the spatial coordinates of pixel $i$ in $Y$. For example, if
$Y$ is of dimension $m \times n$, then $x(i) (x_1(i),x_2(i)) \in
\{1,\ldots,m\} \times \{1,\ldots,n\}$. For each pixel $i \in Y$, let $N_r(i)$
denote the collection of
pixels
\[ 
N_r(i) = \{ j : \|x(i) - x(j)\|_2 < r \},
\]
where $r > 0$ is some small radius, e.g., $r = 2$. We will refer to $N_r(i)$ as
the propagation neighborhood. In addition to $N_r(i)$ we define a larger
calibration neighborhood $C_R(i)$ for $R > r$, which is coarsely sampled.
Specifically, if $(x_1(i),x_2(i)) = x(i)$ denotes the spatial coordinates of
pixel $i$, then we define
\[ 
C_R(i) = \{ j : \|x(j) - x(i)\|_2 <R, x_1(j) - x_1(i) \equiv 0\, (\text{mod }
3), x_2(j) - x_2(i) \equiv 0\, (\text{mod } 3) \},
\]
so that $C_R(i)$ consists of the pixels whose $3 \times 3 \times 3$ feature
cubes are non-overlapping, and whose spatial distance is less than $R$ from
pixel $i$. Using the calibration neighborhood $C_R(i)$ we define
\[
M(i) = \max_{j \in C_R(i)} \|f(i) - f(j)\|_2^2,
\]
and define the asymmetric kernel $W_\varepsilon(i,j)$ by 
$$
W_\varepsilon(i,j) = 
\left\{
\begin{array}{cl}
\exp \left( - \frac{\|f(i) - f(j)\|_2^2}{\varepsilon  
M(i) } \right) & \text{if} \quad j \in N_r(i) \\
0 & \text{if} \quad j \not \in N_r(i).
\end{array} \right.
$$
The purpose of defining the calibration neighborhood $C_R(i)$ rather than just
normalizing the kernel over $N_r(i)$ is to avoid the case where the entire
neighborhood $N_r(i)$ is contained in a single layer. In this case, all pixels
within $N_r(i)$ should be strongly connected, but normalizing using $N_r(i)$
would emphasize minute differences in the feature vectors. Of course, this
problem could be avoided by increasing the radius $r$; however, increasing $r$
increases the number of nonzero entries in each row of the kernel $W$ by a
factor of $r^2$: better to define a separate calibration neighborhood $C_R(i)$
to avoid both the normalization and complexity issues.  As defined, the kernel
$W_\varepsilon(i,j)$ takes a maximum value of $1$ when $f(i) = f(j)$, and a
minimum value of 
\[
W_\varepsilon(i,j) = \exp \left( - \frac{1}{\varepsilon} \right)
\]
if $\|f(i) - f(j)\|_2^2 = M(i)$. Therefore, rather than choosing $\varepsilon
>0$ we parameterize our model by $\delta >0$ and define
\[
\varepsilon = - \frac{1}{\log \delta}
\]
With this parameterization, the minimum possible value of the kernel in each
neighborhood is $\delta$.  In order to be compatible with the standard Diffusion
Maps framework described in \cite{Coifman:2006}, the kernel used to define the
diffusion process must be symmetric. Therefore, we define a symmetrized version
$K(i,j)$ of $W_\varepsilon(i,j)$ by 
\[
K(i,j) = \frac{1}{2} \left( W_\varepsilon(i,j) + W_\varepsilon(j,i) \right) 
\]
Notice that the dependence on $\varepsilon$ of $K$ (which in turn depends on
$\delta$) was suppressed for notational simplicity. 

\subsubsection*{Summary of Kernel Construction} We first choose a radius $r > 0$ which defines a
spatially local propagation neighborhood. Second, we choose a radius $R > 0$
which defines a coarse calibration neighborhood. Third, we choose a parameter
$\delta >0$, which determines the minimum possible affinity in each
neighborhood. Finally, we symmetrize the constructed kernel, resulting in the
symmetric kernel $K(i,j)$.  In the following section, we use the kernel $K$ in
combination with the standard Diffusion Maps framework \cite{Coifman:2006} to
create a description of the layer geometry of the image $Y$.

\subsection{Layer Organization}
In this section, we describe how the kernel $K$ defined in the previous section
can be used to generate a description of the layer structure of $Y$. In
particular, we use the $(\alpha =0)$ diffusion kernel construction of
Diffusion Maps \cite{Coifman:2006} in combination with the randomized matrix
decomposition technique as introduced by \cite{Woolfe:2008},
as well as intuition about the form of the Laplacian-Neumann eigenvectors, to
define a new depth function $h$ of the pixels of the seismic image $Y$.

Given the kernel $K(i,j)$, we define the diffusion operator $P$ by
\[
P(i,j) = \frac{K(i,j)}{q(i)},
\]
where
\[ 
q(i) = \sum_{j} K(i,j).
\]
The matrix operator $P(i,j)$ is row stochastic; therefore, $P$ is a Markov
operator when applied from the right, and a diffusion or averaging operator when
applied from the left. Since $P$ is similar to the symmetric matrix $A$,
\[
A = Q^{1/2} P Q^{-1/2},
\]
where
\[ 
Q = \diag(q),
\]
the matrix $P$ is diagonalizable and has real eigenvalues. Moreover, the
eigenvectors $\psi_j$ of $P$ are exactly given by
\[
\psi_j = Q^{-1/2} \phi_j
\]
where $\phi_j$ is an eigenvector of $A$ of the same eigenvalue $\lambda_j$.

Therefore, in order to compute the eigenvectors of $P$, it suffices to decompose
$A$ and apply $Q^{-1/2}$. From a computational point of view, working with $A$
is preferable since $A$ is symmetric and has operator norm $1$. Furthermore, the
operator $A$ is sparse by construction with most $2  |N_r(i)|$ nonzero
entries in each row $i$, and is of finite rank invariant to the sampling density
of the data (at least when when the kernel bandwidth is fixed). Therefore,
randomized matrix decomposition algorithms, such as described in
\cite{Halko:2009}, provide an efficient method to compute the top few
eigenvectors of $A$. Given these eigenvectors, the eigenvectors of $P$ can be
recovered by applying $Q^{-1/2}$ as previously mentioned.

By construction, the eigenvectors of $P$ approximate the Neumann Laplacian
eigenfunctions of the underlying geometry of the pixels of $Y$ on which $P$
corresponds to isotropic diffusion, cf. \cite{Coifman:2006}. The largest
eigenvalue $\lambda_0$ of $P$ is $1$, which corresponds to a constant
eigenvector $\psi_0 = 1$. However, the first nontrivial eigenvector $\psi_1$ is
of interest. Recall that $P$ is based on the kernel $K$, which is defined to be
large (i.e., close to $1$) between similar pixels within a
spatial neighborhood, and on the order of a small constant $\delta$ (e.g.,
$\delta = 10^{-7}$) between highly dissimilar pixels with respect to the
calibration neighborhood. While moving along a layer in the seismic image $Y$,
the pixels (as characterized by the feature vector $f(i)$) should be highly
similar.  However, when moving perpendicular to a layer, the feature vectors
$f(i)$ will change rapidly, and pixels quickly become highly dissimilar to the
starting point.  As a result, the diffusion described by $P$ travels rapidly
along the layers of $Y$, and slowly perpendicular to these layers. That is to
say, the underlying geometry of the pixels on which $P$ represents isotropic
diffusion must have a relatively small distance between pixels on the same layer
when compared to pixels on different layers, which are equally spatially
distant in $Y$.

Therefore, the underlying geometry of $P$ roughly corresponds to a tall
rectangle shape, that is, a rectangle whose heigh is greater than its width. Of
course, due to the variation in the earth, the underlying domain will not be
exactly rectangle shaped, but resemble some type of deformed tall rectangle
where the layers run horizontally and are relatively flat. In Section
\ref{sec:proof}, we proved an explicit stability result, which indicates that
the first Neumann Laplacian eigenfunctions on tall rectangle are highly
stable under deformations on the underlying domain. This proof is consistent
with numerical experiments. Intuitively, the stability results from the fact
that if a rectangles height is a least twice its width, the first couple Neumann
Laplacian eigenfunctions will only oscillate vertically, which provides a buffer
to any horizontal oscillations.

Recall, that the first Neumann Laplacian eigenfunction on the rectangle
$[0,\varepsilon]
\times [0,1]$ where $1 > \varepsilon > 0$ is
\[
\psi_1(x,y) = \cos( \pi y )
\]
Assuming that $\psi_1$ is of a similar form, the height function on the
underlying domain can be recovered by
\[
h(i) = \arccos \left( \frac{\psi_1(i) - \min_j \psi_1(j)}{ \max_j \psi_1(j)
- \min_j \psi_1(j)} \right)
\]
up to linear scaling. Depending on the height of the underlying rectangular
domain, the next couple eigenfunctions $\psi_2$, $\psi_3$, etc. should resemble
$\cos \pi 2 y$, $\cos \pi 3 y$, etc., until the first eigenvector emerges with
horizontal oscillations, after which the 
eigenfunctions likely become difficult to interpret because of noise. 
\subsubsection*{Summary of layer organization} We have constructed a diffusion operator $P$ whose
eigenfunctions approximate the Neumann Laplacian eigenfunctions on the
underlying domain, which by the construction of the kernel $K$ should resemble a
tall rectangle. Using randomized matrix decomposition techniques, the
eigenfunctions of $P$ can be computed, and the $\arccos$ function can be used to
recover the height function of the underlying rectangular domain. This height
function should organize the layers of the image.

\subsection*{Acknowledgements}
We would like to thank Anthony Vassiliou and Lionel Woog for the seismic data
and for bringing the application problem to our attention. Furthermore, we would
like to thank Ronald Coifman for numerous insightful conversations, and Stefan
Steinerberger for many helpful comments and suggestions.


\begin{thebibliography}{10}

\bibitem{Lafon:2004}
S.~Lafon, {\em Diffusion Maps and Geometric Harmonics}.
\newblock PhD thesis, Yale University, 2004.

\bibitem{Lederman:2014}
R.~R. Lederman and R.~Talmon. 
\newblock Common manifold learning using alternating-diffusion.
\newblock {\em Yale Tech. Rep.}, 2014.

\bibitem{Hirn:2014}
R.~R. Coifman and M.~J. Hirn. 
\newblock Diffusion maps for changing data.
\newblock {\em Appl. Comput. Harmon. Anal.}, 36(1):79--107, 2014.

\bibitem{Talmon:2013b}
R.~Talmon and R.~R. Coifman.
\newblock Empirical intrinsic geometry for nonlinear modeling and time series
filtering. 
\newblock {\em Proc. Natl. Acad. Sci. USA}, 110:12535--12540, 2013.

\bibitem{Singer:2011}
A.~Singer and H.~Wu.
\newblock Vector diffusion maps and the connection laplacian.
\newblock {\em Comm. Pure Appl. Math}, 65(8):1067--1144, 2012.

\bibitem{Schclar:2010}
A.~Schclar, A.~Averbuch, N.~Rabin, V.~Zheludev, and K.~Hochman. 
\newblock A diffusion framework for detection of moving vehicles.
\newblock {\em Digit. Signal Process.}, 20(1):111--122, 2010.

\bibitem{Plessis:2009}
L.~du~Plessis, S.~Damelin, and M.~Sears. 
\newblock Reducing the dimensionality of hyperspectral data using diffusion
maps.  
\newblock {\em Geoscience and Remote Sensing Symposium IEEE International
IGARS}, 4:885--888, 2009.

\bibitem{Belkin:2003b}
M.~Belkin and P.~Niyogi.
\newblock Laplacian eigenmaps for dimensionality reduction and data
representation.  
\newblock {\em Neural Comput.}, 15:1373--1396, 2003.

\bibitem{Lomask:2007}
J.~Lomask. 
\newblock {\em Seismic Volumetric Flattening and Segmentation}.
\newblock PhD thesis, Stanford University, 2007.

\bibitem{Gao:2009}
D.~Gao.
\newblock 3d seismic volume visualization and interpretation: An integrated
workflow with case studies.
\newblock {\em Geophysics}, 74(1):1--12, 2009.

\bibitem{Parks:2010}
D.~Parks.
\newblock {\em Seismic image flattening as a linear inverse problem.} 
\newblock Master's thesis, Colorado School of Mines, 2010.

\bibitem{Coifman:2006}
R.~R. Coifman and S.~Lafon.
\newblock Diffusion maps. 
\newblock {\em Appl. Comput. Harmon. Anal.}, 21(1):5--30, 2006.

\bibitem{Woolfe:2008}
F.~Woolfe, E.~Liberty, V.~Rokhlin, and M.~Tygert.
\newblock A fast randomized algorithm for the approximation of matrices.
\newblock {\em Appl. Comput. Harmon. Anal.}, 25(3):335--366, 2008.

\bibitem{Halko:2009}
N.~{Halko}, P.-G.~{Martinsson}, and J.~A.~{Tropp}.
\newblock Finding structure with randomness: Probabilistic algorithms for
constructing approximate matrix decompositions.
\newblock {\em ArXiv e-prints}, 2009.

\end{thebibliography}
\end{document}